\documentclass[11pt,a4paper,reqno]{amsart}

\usepackage[colorlinks=false]{hyperref}
\usepackage{amsmath, nccmath, amsthm, amsfonts, amssymb, mathtools}
\usepackage{tikz,tikz-cd}
\usepackage[utf8]{inputenc}
\usepackage[T1]{fontenc}
\usepackage[english]{babel}
\usepackage{sourcesanspro}
\usepackage[theoremfont, amsthm]{newtx}
\usepackage{BOONDOX-calo,BOONDOX-ds,BOONDOX-frak}
\usepackage{enumitem}
\usepackage{comment}
\usepackage{cleveref}
\usepackage{xcolor}
\usepackage{soul}
\sethlcolor{yellow!50} % Set highlight color
\theoremstyle{plain}
\newtheorem{theorem}{Theorem}%[section]
\newtheorem{proposition}[theorem]{Proposition}
\newtheorem{lemma}[theorem]{Lemma}
\newtheorem{corollary}[theorem]{Corollary}
\newtheorem{definition}[theorem]{Definition}
\theoremstyle{remark}
\newtheorem{remark}[theorem]{Remark}
\newtheorem{example}[theorem]{Example}
\numberwithin{equation}{section}
\newcommand{\al}[1]{\mathbf{#1}}
\newcommand{\var}[1]{\mathcal{#1}}
\newcommand{\vag}[1]{\mathrm{V}\!\left({#1}\right)}
\newcommand{\trop}[1]{\mathrm{#1}}

\newcommand{\hto}{\mathrm{h}}
\newcommand{\bro}{\mathrm{b}}
\newcommand{\meo}{\mathrm{m}}

\newcommand{\acr}{\newline\indent}
\author[V\'aclav Cenker]{V\'aclav Cenker*}
\title{Finitely Generated Varieties of Commutative \linebreak BCK-algebras:  Covers}
\address{\llap{*\,}
Palacký University Olomouc\acr
Faculty of Science\acr
Department of Algebra and Geometry\acr
17. listopadu 1192/12\acr
CZ--779 00 Olomouc\acr
Czech Republic}
\email{vaclav.cenker01@upol.cz}

\thanks{Author acknowledge the support by the Czech Science Foundation (GA\v CR): project 24-14386L}

\begin{document}
\begin{abstract}
The article aims to describe all covers of every finitely generated variety of cBCK-algebras. 
It is known that subdirectly irreducible cBCK-algebras are rooted trees (with respect to their order). 
Moreover, all subdirectly irreducible members of a finitely generated variety are isomorphic to subalgebras of the subdirectly irreducible generators of that variety. 
The first part of the article focuses on subalgebras of finite subdirectly irreducible cBCK-algebras. 
The second and main part presents a construction that provides all the covers of any finitely generated variety, which constitutes the central contribution of the article.
\end{abstract}
\maketitle

\section{Introduction}
BCK-algebras were first introduced in \cite{Isek} as an algebraic semantics for a non-classical logic based solely on implication. Every BCK-algebra admits a natural ordering, and if it satisfies commutativity (which in a BCK-algebra is not the same as standard commutativity of a binary operation), then the underlying poset forms a meet semilattice. For brevity, we will refer to BCK-algebras satisfying this property as cBCK-algebras. In the literature, this commutativity is sometimes called Tanaka’s identity.

Unlike general BCK-algebras, cBCK-algebras form a variety. While our main focus is on cBCK-algebras, it is worth noting that this variety is properly contained within the larger variety of JBCK-algebras, alongside other notable varieties such as the implicative subreducts of hoops or Hilbert algebras. In this context, the “J” in JBCK refers to an identity that generalises Tanaka’s identity; see~\cite{BlokRaf} for details. For further information on the subvarieties of JBCK-algebras, see~\cite{Agliano}.

The class of commutative BCK-algebras is also closely related to MV-algebras. Every bounded cBCK-algebra can be turned into an MV-algebra, and conversely, by taking the implicative subreducts of an MV-algebra, one obtains a cBCK-algebra. This construction yields a more special class of BCK-algebras, known as Łukasiewicz BCK-algebras (\L BCK-algebras). \L BCK-algebras form a subvariety of cBCK-algebras that additionally satisfies prelinearity, and they are equivalently referred to as cBCK-algebras with the relative cancellation property; this terminology is used, for instance, in \cite{DvurBCK}. A notable fact is that every linearly ordered cBCK-algebra is automatically a \L BCK-algebra. For more details on MV-algebras, the reader may consult \cite{CignoliMV}.

Focussing on cBCK-algebras, we first recall some of their fundamental structural properties. The variety of all cBCK-algebras is congruence distributive and 3-permutable, while no non-trivial subvariety is 2-permutable. Finitely generated varieties of cBCK-algebras are semisimple, i.e., any subdirectly irreducible member is simple. Moreover, every finite simple cBCK-algebra is hereditarily simple. Finally, an important structural fact is that subdirectly irreducible cBCK-algebras can be viewed, with respect to their order, as rooted trees \cite{Romanowska, CornishTrees}.

Keeping these structural properties in mind, we now proceed to outline the method used to describe the covers of finitely generated varieties of cBCK-algebras.
Let $\var{V}$ be a finitely generated variety of cBCK-algebras. Then, there exists an irredundant set of finite subdirectly irreducible cBCK-algebras $\al A_1$, \dots $\al A_n$ such that $\var{V} = \vag{\al A_1,  \dots, \al A_n}$. A crucial observation is that the subdirectly irreducible members of $\var{V}$, $\trop{Si}(\var{V})$, consist (up to isomorphism) precisely of the subalgebras of $\al A_1$, \dots, $\al A_n$. Thus,
$$\trop{Si}(\var{V}) = \bigcup_{i=1}^n \trop{IS}(\al A_i).$$
This follows from Jónsson's Lemma together with the fact that $\al A$ is hereditarily simple.

Because subdirectly irreducible members of $\var{V}$ are essentially subalgebras of generators, we first examine subalgebras of finite subdirectly irreducible cBCK-algebras. Given a finite subdirectly irreducible cBCK-algebra $\al A$, subalgebras of $\al A$ can be divided into two types: downsets and others. The latter can be characterised as the set of elements of $\al A$ whose heights are divisible by some $k > 1$, $k \in \mathbb{N}$. Under suitable conditions, such a set indeed forms a subalgebra. A detailed characterisation of these subalgebras constitutes the first part of this article.

We then turn to the central part of the article, which focusses on identifying subdirectly irreducible cBCK-algebras that generate a cover of $\vag{\al A_1, \dots, \al A_n}$. 
The construction formalises a straightforward idea: Given a subdirectly irreducible cBCK-algebra, a tree, one can generate a minimal algebra above it by extending the tree. 
This begins by adding a new leaf to an existing leaf, providing a simple visual guide to the process.

\section{Basic definitions}

A \emph{BCK-algebra} is an algebra $\al A = (A, \ominus , 0)$ of type $(2,0)$ such that the following identities and quasi-identity hold:
\begin{gather}
((x \ominus y) \ominus (x \ominus z)) \ominus (z \ominus y) = 0 \text{,} \label{id: BCK1}\\
x\ominus 0= x \label{id: BCK2}\text{,}\\
0 \ominus x = 0 \label{id: BCK3}\text{,}\\
x \ominus y = 0 \text{ and } y \ominus x = 0 \text{ imply } x = y\text{.}
\end{gather}
In every BCK-algebra, we define an order by
$$x \leq y \text{ if and only if } x \ominus y = 0\text{.}$$
Clearly, the constant $0$ is the bottom element. Further, we write $x \prec y$ to indicate that $y$ covers $x$.

A \emph{commutative BCK-algebra} (\emph{cBCK-algebra} in short) is a BCK-algebra that additionally satisfies the identity
\begin{gather}
x \ominus (x \ominus y) = y \ominus (y \ominus x)\text{.} \label{id: commutativity}
\end{gather}
The class of all cBCK-algebras forms a variety and can be axiomatised by the identities \labelcref{id: BCK2}, \labelcref{id: BCK3}, \labelcref{id: commutativity} and the identity
\begin{align}
(x \ominus y) \ominus z &= (x \ominus z) \ominus y\text{.}\label{id: BCKex}
\end{align}
The identity \labelcref{id: commutativity} makes the order a semilattice, where meet is
$$x \wedge y = x \ominus (x \ominus y)\text{.}$$
It is a straightforward fact, but it may be useful to keep in mind during computations that $x \ominus y = x \ominus (x \wedge y)$.

In any BCK-algebra, we define
\begin{gather*}
x \ominus 1y = x \ominus y\text{,} \quad x\ominus (n+1)y = (x \ominus ny)\ominus y\text{, } \; n \in \mathbb{N}\text{.}
\end{gather*}
For technical purposes, we also define $x \ominus 0y = x$.

A subclass of particular interest is the class of so-called \L BCK-algebras.  
They form a variety axiomatised, relative to cBCK-algebras, by the identity
\begin{align*}
    (x \ominus y) \ominus (y \ominus x) = x \ominus y,
\end{align*}
or equivalently,
$$
(x \ominus y) \wedge (y \ominus x) = 0.
$$
This identity is a form of prelinearity and ensures that subdirectly irreducible members are linearly ordered. Moreover, every cBCK-chain is necessarily a \L BCK-chain.

For the purposes of this article, the most important examples of \L BCK-chains are the finite subalgebras of the \L BCK-chain $(\mathbb{Z}^+, \ominus, 0)$.  
Note that $(\mathbb{Z}^+, \ominus, 0)$ itself is a subalgebra of $(\mathbb{R}^+, \ominus, 0)$,  
where the operation is defined by
$$
x \ominus y = \max\{x - y, 0\}.
$$

The finite subalgebras of $\mathbb{Z}^+$ that are intervals $\al{[0, n]}$ with $n \in \mathbb{N}$ will be denoted by $\al S_n$, that is, $S_n = \{0, 1, \dots, n\}$.

The variety of \L BCK-algebras is closely related to MV-algebras.  
More precisely, every \L BCK-algebra is a subreduct of an MV-algebra, and conversely any bounded \L BCK-algebra can be expanded to an MV-algebra; this holds in particular for finite cBCK-chains.  
Although introducing MV-algebras is not essential for the present paper, mentioning them provides useful context.  
Readers seeking further details, especially on subdirectly irreducible members of this class, may consult \cite{DvurBCK} and \cite{CignoliMV}.

\section{Subdirectly irreducible algebras and their subalgebras}

\subsection{Finite trees are cBCK-algebras}
In what follows, by a tree, we mean a lower semilattice $(A, \wedge, 0)$ with a bottom element $0$, in which any interval $[0,a]$ is a chain. Thus, two elements in the tree have an upper bound if and only if they are comparable.

As mentioned earlier, there is a strong connection between subdirectly irreducible cBCK-algebras and trees. 
We now state this explicitly. 
For more details, see \cite[Theorem~3.1]{Romanowska} 
or its variant in \cite[Theorem~5]{Pal}. 

\begin{theorem}
A cBCK-algebra is subdirectly irreducible if and only if the underlying semilattice is a tree with meet irreducible $0$ and such that every maximal chain is subdirectly irreducible cBCK-chain.
\end{theorem}

Let us remark that if we consider a finite cBCK-algebra, then the condition of $0$ being meet irreducible may be replaced by the condition that the tree has a single atom. 

While in general the BCK-operation of a subdirectly 
irreducible cBCK-algebra is not uniquely determined by its underlying tree semilattice, this does hold in the finite case. 
In fact, \cite{CornishTrees} establishes an even stronger result: Any tree satisfying the descending chain condition determines a unique way of defining the BCK-operation. 
Since we are concerned with finitely generated varieties, 
we only need the following corollary of \cite[Theorem 1.4]{CornishTrees}.

\begin{corollary} \label{cor: trees are cbck}
For any finite tree $(A, \wedge, 0)$ with a single atom, there is, up to isomorphism, a unique subdirectly irreducible cBCK-algebra $\al{A'}$ whose underlying semillatice is isomorphic to $(A, \wedge, 0)$.
\end{corollary}

\begin{remark}
This corollary allows us to present examples of finite subdirectly irreducible cBCK-algebras simply by giving the Hasse diagram of their underlying trees.
\end{remark}

The reader should wonder how the BCK-operation can be defined once the underlying tree is known. Although \cite{CornishTrees} provides an answer, we do not reproduce it verbatim. Instead, we start with finite chains, where the definable BCK-operation arises very naturally, and based on this we then show how to define the BCK-operation on finite trees.

Let us recall a few facts about subdirectly irreducible cBCK-chains. Every cBCK-chain is an \L BCK-algebra. 
Hence, a subdirectly irreducible cBCK-chain is necessarily a subdirectly irreducible \L BCK-chain. 
Moreover, every finite \L BCK-chain is simple, 
and every finite simple \L BCK-algebra is a chain. 
Furthermore, any finite \L BCK-chain is isomorphic to $\al{S}_n$ for some $n \in \mathbb{N}$. 
If $n \neq m$, then $\al{S}_n$ is not isomorphic to $\al{S}_m$. 
Thus, for each finite cBCK-chain $\al{A}$, there exists a unique $n \in \mathbb{N}$ such that $\al{A} \cong \al{S}_n$. All of these results are well established (see~\cite{CornishBCK, Komori}); for an extensive overview of the research on \L BCK-algebras, one may also consult~\cite{DvurBCK}.

With this in mind, we may view finite subdirectly irreducible 
cBCK-algebras simply as trees in which each maximal chain $[0,a]$ 
is (isomorphic to) some $\al S_n$. 
On the other hand, given any finite chain $(C, \wedge, 0)$, 
there is only one way of defining the BCK-operation, 
because having two distinct BCK-operations would result in two cBCK-algebras that (according to the previous paragraph) are isomorphic.

Let $(A, \wedge, 0)$ be a finite tree, and let $(C_i, \wedge, 0)$, $i \in I$, denote all its maximal chains.\footnote{Clearly, the maximal chains cover the tree.} For each $i$, there is a unique way of defining the cBCK-chain $\al C_i$, and there exists an isomorphism $\iota_i \colon \al C_i \to \al S_{n_i}$. To define the BCK-operation on the entire $(A, \wedge, 0)$, we proceed as follows. For any $a,b \in A$, there is an $i$ such that $a \in C_i$, and we define
$$a \ominus b := \iota^{-1}_i(\iota_i(a) \ominus \iota_i (a \wedge b))\text{.}$$
In summary, the guiding idea is quite natural: 
the BCK-operation in a finite tree behaves like truncated subtraction on $\mathbb{Z}^+$. The only adjustment is that, when subtracting incomparable elements, we first project the second element (via the meet) onto the chain determined by the first element.

\subsection{Subalgebras of finite subdirectly irreducible cBCK-algebras}

Recall that any finite subdirectly irreducible cBCK-algebra is hereditarily simple. Thus, by studying its subalgebras, we are in fact examining the subdirectly irreducible members of the variety generated by that cBCK-algebra.

Before proceeding, we introduce and clarify some notions.
Let $\al A$ be a subdirectly irreducible cBCK-algebra. Then:
\begin{itemize}
    \item We use $\prec$ to denote the covering relation, while $<$ denotes the strict order.
    \item $\bro (\al A)$ denotes a set of all branching elements of $\al A$. It holds that $b \in \bro(\al A)$ iff there exist $c$, $d \in A$ such that $c$, $d > b$ and $c \wedge d = b$.
    \item For $a \in A$, $\hto_A(a) = \vert [0, a]\vert -1$ denotes the height of the element $a$. If there is no danger of confusion, we write $\hto(a)$ instead of $\hto_A(a)$. Furthermore, $\hto(\al A) = \sup \{\hto(a) \mid a \in A \}$ denotes the height of $\al A$. 
    \item $a \parallel b$ denotes that $a$ is incomparable to $b$, and $a \not\parallel b$ indicates that $a$ is comparable to $b$.
    \item By the width of $\al A$ we mean the cardinality of a maximal antichain (a maximal set of mutually incomparable elements) of $\al A$.
    \item $\meo (\al A)$ denotes a set of all maximal elements of $\al A$. For $\al A$ with finite height, $\vert \meo(\al A)\vert $ is a width of $\al A$.
\end{itemize}

The following will be used in the proofs without hesitation. 

\begin{lemma}\label{lem: height and minus}
    Let $\al A$ be a finite subdirectly irreducible cBCK-algebra. Let $a$, $b \in A$ and $n \in \mathbb N$. Then,
    \begin{gather*}
        \hto(a \ominus nb) = \max\{\hto(a) - n\hto(a \wedge b),\, 0\}\text{.}
    \end{gather*}
    Moreover, if $e$ is an atom of $\al A$ and $a \neq 0$, then $\hto(a)$ is the least $n \in \mathbb{N}$ such that $a \ominus ne = 0$.
\end{lemma}
\begin{proof}
The first part is borrowed from \cite[Lemma 2.1]{CornishTrees}. The moreover part is obtained as follows.  Let $e \in A$ be an atom of $\al A$. For any $a \in A$, $a \neq 0$, the sequence $(a \ominus n e)_{n \in \mathbb{N}}$ is decreasing and eventually reaches $0$. Therefore, there must exist the least $n \in \mathbb{N}$ for which $a \ominus ne = 0$. Let $n$ denote this natural number. If $\hto(a) > n$, then $\hto(a) - n = \hto(a) - n \hto(e) > 0$. By the first part of the lemma, it follows that $\hto(a \ominus n e) > 0$, which is impossible. Therefore, we must have $\hto(a) \leq n$.  

On the other hand, if $h(a) < n$, then $\hto(a) - n = \hto(a) - n \hto(e) < 0$, so that $\hto(a) - (n-1) \hto(e) \leq 0$, whence $\hto(a \ominus (n-1) e) = 0$, and finally $a \ominus (n-1) e = 0$, since $0$ is the only element of height $0$. This contradicts $n$ being the least and hence $\hto(a) = n$.
\end{proof}

Although the following fact is evident, we state it explicitly, as it will be used repeatedly in the article.

\begin{lemma}
In the finite cBCK-chain, every element is uniquely determined by its height.
\end{lemma}

The first simple fact we can state about finite subdirectly irreducible cBCK-algebras concerns their generators. 
In fact, to generate the algebra, it is necessary to take the maximal elements since the BCK-operation is decreasing. 
Although this is enough in some cases, in general, to ensure that the entire algebra is generated, we also need to include an atom.

\begin{lemma}\label{lem: gen A}
Let $\al A$ be a finite subdirectly irreducible cBCK-algebra. Then $\al A$ is generated by $\meo(\al A) \cup \{a\}$, where $a$ is an atom of $\al A$.
\end{lemma}
\begin{proof}
Let $\al A$ be a finite chain. Let $a$ be an atom of $\al A$ and $c$ the maximal element of $\al A$, then for any $b \in A$
$$\hto(b) =\hto(c) - \big(\hto(c) - \hto(b)\big)\hto(a).$$
By the Lemma \labelcref{lem: height and minus}, we have
$\hto(b) = \hto(c \ominus \big(\hto(c) - \hto(b)\big) a)$. Since each element is uniquely determined by its height it follows that $b= c \ominus \big(\hto(c) - \hto(b)\big) a$.

Now, let $\al A$ be a finite subdirectly irreducible cBCK-algebra. Let $a$ be an atom. For any $b \in \al A$, there exists $c \in \meo(\al A)$ such that $b \in [0, c]$. From the first paragraph, it follows that $b = c \ominus k a$ for some $k \in \mathbb{N} \cup \{ 0 \}$.
\end{proof}

The fact that $\meo(\al A)$ may not suffice to generate the entire $\al A$ is closely related to the description of its subalgebras. If $\meo(\al A)$ is not enough -- that is, if the generating process does not reach an atom of $\al A$ -- then the subalgebra generated by $\meo(\al A)$ is not a downset of $\al A$.  This suggests that besides the downset subalgebras, there exist other kind of subalgebras as well. We now provide a characterisation of these subalgebras.

Let $\al A$ be a finite subdirectly irreducible cBCK-algebra. Let $k \in \mathbb{N}$. Denote
$$
A_k = \{a \in A \mid k \text{ divides } \hto(a)\}\text{.}
$$
If $\al A$ is a chain, then $A_k$ forms the universe of a subalgebra of $\al A$ for any $A_k$. 
However, if $\al A$ is not a chain, $A_k$ does not need to be the universe of a subalgebra. 
As we shall see, this depends on whether $\bro(\al A) \subseteq A_k$. 
In fact, we will rather consider $A_k \cup \meo(\al A)$ instead of $A_k$ alone, 
which allows for slightly better handling later on.

\begin{remark}
If $\al A$ has height $n$ and $k > n$, then $A_k$ is, by definition, just $\{0\}$. On the other hand, if $k \leq n$, then $A_k$ contains at least two elements. Furthermore, if $k = 1$, then, of course, $A_k = A$.
\end{remark}

\begin{figure}[t]
    \centering
    \includegraphics[height = .18\textheight]{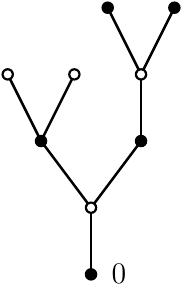}
    \caption{A subdirectly irreducible cBCK-algebra and the set $A_2$ (marked with $\bullet$).}
    \label{fig: 3}
\end{figure}

In the following, $\gcd(k, l)$ denotes the greatest common divisor of $k$, $l \in \mathbb{Z}$.

\begin{lemma}\label{lem: gcd}
Let $\al C$ be a finite cBCK-chain. Let $a$, $b \in C$ be such that $\gcd(\hto(a), \hto(b)) = d$. Then there are a BCK-term $t(x,y)$ and a unique element $e$ such that $t(a,b)= e$ and $\hto(e)= d$.
\end{lemma}
\begin{proof}

For $a = 0$ or $b = 0$, the lemma obviously holds. Therefore, assume that $a \neq 0$ and $b \neq 0$.

Since in a finite cBCK-chain $\al C$ each element is uniquely determined by its height, and every $n \in \mathbb{N}$ with $n \le \hto(\al C)$ is the height of some element, it follows that there exists a unique element $e \in C$ such that $\hto(e) = d$.

Without loss of generality, assume $\hto(a) > \hto(b)$. To construct a term $t(x,y)$, consider the Euclidean algorithm applied to $\hto(a)$ and $\hto(b)$. The algorithm corresponds to the following equalities:
\begin{align*}
    \hto(a) - q_0 \hto(b) &= r_0,\\
    \hto(b) - q_1 r_0 &= r_1,\\
    &\vdots\\
    r_{n-2} - q_n r_{n-1} &= r_n = d,
\end{align*}
where $q_0, \dots, q_n$ and $r_0, \dots, r_n$ are all positive integers.

Since $r_0 > 0$, we have
$$r_0 = \hto(a) - q_0 \hto(b) = \hto(a \ominus q_0 b).$$
Abbreviate $c_0(a,b) = a \ominus q_0 b$. Then, the second line of the algorithm gives
$$r_1 = \hto(b) - q_1 \hto(c_0(a,b)) = \hto(b \ominus q_1 c_0(a,b)),$$
and we denote $c_1(a,b) = b \ominus q_1 c_0(a,b)$.  

Clearly, the same applies to every subsequent line. In the final line, we obtain
$$d = \hto(e) = \hto(c_{n-2}(a,b) \ominus q_n c_{n-1}(a,b)).$$
Thus, we can define
$$t(x,y) = c_{n-2}(x,y) \ominus q_n c_{n-1}(x,y).$$
As explained earlier, $e$ is uniquely determined by its height. Therefore, $e = t(a,b)$.
\end{proof}

\begin{proposition}\label{prop: Ak def sub}
Let $\al A$ be a finite subdirectly irreducible cBCK-algebra. Let $A_k$ has the same meaning as above, and let $A_k \neq \{0\}$. Then $ A_k \cup \meo(\al A)$ is a universe of a subalgebra of $\al A$ if and only if $\bro(\al A)$, $ \meo(\al A) \subseteq A_k$.
\end{proposition}
\begin{proof}
($\Rightarrow$) Let $A_k \cup \meo(\al A)$ be the universe of a subalgebra of $\al A$. Let $M \subseteq \meo(\al A)$ be any set of maximal elements. Then, $\bigwedge_{m \in M} m \in A_k \cup \meo(\al A)$. Clearly, for any $b \in \bro(\al A)$, we can find $M \subseteq \meo(\al A)$ such that $b = \bigwedge_{m \in M} m$. One possible choice is $M = \{ m \in \meo(\al A) \mid m > b \}$. Therefore, since $A_k \cup \meo (\al A)$ is closed under $\ominus$ and hence under $\wedge$, we have $\bro(\al A) \subseteq A_k \cup \meo(\al A)$. Moreover, for any $b \in \bro(\al A)$, we have $b \notin \meo(\al A)$. It follows that $\bro(\al A) \subseteq A_k$.

Assume that $\meo(\al A) \not\subseteq A_k$. Suppose first that $\al A$ is a chain. Then $\meo(\al A)$ is a singleton $\{m\}$ with $m \notin A_k$. Furthermore, since $A_k \neq \{0\}$, there exists $a \in A_k$, $a \neq 0$. By Lemma \labelcref{lem: gcd}, the algebra $A_k \cup \meo(\al A)$ should include an element $e$ with height $\gcd(\hto(m), \hto(a)) = d$. However, $d$ cannot be divisible by $k$ because $\hto(m)$ is not divisible by $k$. Moreover, $e \neq m$. Consequently, $e \notin A_k \cup \meo(\al A)$, which is a contradiction.

Now, suppose that $\al A$ is not a chain. Then $\bro(\al A) \neq \emptyset$. Therefore, each maximal chain $[0, m]$, $m \in \meo(\al A)$, includes at least one branching element (for example, the smallest). Let $m \in \meo(\al A)$ be such that $m \notin A_k$, and let $\al{[0, m]}$ be the chain subalgebra of $\al{A_k \cup \meo(\al A)}$. There is some branching element $b$ such that $b \in [0, m]$. From the first paragraph of the proof, we know that $\hto(b)$ is divisible by $k$. Therefore, by an argument similar to the previous paragraph, the algebra $\al{[0, m]}$ contains an element $e$ of height $\gcd(\hto(b), \hto(m)) = d$. Since $d$ cannot be divisible by $k$, we obtain a contradiction.

($\Leftarrow$) Let $\bro(\al A), \meo(\al A) \subseteq A_k$. Let $a$ and $b$ be arbitrary elements of $A_k$. If $b = 0$, then $a \ominus b = a \in A_k$. Furthermore, if $a \leq b$, then $a \ominus b = 0 \in A_k$. Hence, we may assume $b \neq 0$ and $a \ominus b \neq 0$.
Moreover, assume that $a \not\parallel b$ (so that $a > b$). Since $a, b \in A_k$, we have $\hto(a) = kp$ and $\hto(b) = kq$ for suitable $p, q \in \mathbb{N}$. It follows that $\hto(a \ominus b) = kp - kq = k(p-q)$. Therefore, $a \ominus b \in A_k$.
Now, consider the case of $a \parallel b$. Then, $a \ominus b = a \ominus (a \wedge b)$, where $a \wedge b \in \bro(\al A) \subset A_k$. Here, $a$ and $a \wedge b$ are comparable and both belong to $A_k$. Using the same argument as above, $a \ominus b \in A_k$. Consequently, $A_k$ is closed under $\ominus$.
\end{proof}

Let $\al A$ be a finite subdirectly irreducible cBCK-algebra.
We denote by $\trop S_{\delta}(\al A)$ the set of subalgebras of the form $A_k \cup \meo(\al A)$, where $k \neq 1$.
Furthermore, let $\trop S_d(\al A)$ denote the subalgebras of $\al A$ that are downsets.

\begin{proposition}\label{prop: gcdab}
Let $\al A$ be a finite subdirectly irreducible cBCK-algebra. Let $\al B \in \trop S(\al A)$. If there exists $a$, $b \in \al B$ such that $\gcd (\hto_A(a), \hto_A(b)) = 1$, then $\al B \in \trop S_d(\al A)$.
\end{proposition}
\begin{proof}
Let $\gcd (\hto_A(a), \hto_A(b)) = 1$. Assume that $a > b$ or $a < b$. Then, by Lemma \labelcref{lem: gcd}, there exists a term $t(x,y)$ such that $t(a,b)$ is an atom of $A$.

If, on the other hand, $a \parallel b$, then $a \wedge b < a, b$. Now, if $\gcd(\hto(a \wedge b), \hto(a))= 1$ or $\gcd(\hto(a \wedge b), \hto(b))= 1$, we may use the same argument as before to obtain an atom. Thus, assume the opposite, i.e., $\gcd(\hto(a \wedge b), \hto(a))= d_1$ and $\gcd(\hto(a \wedge b), \hto(b))= d_2$, where $d_1 \neq 1$ and $d_2 \neq 1$. It follows that $\gcd(d_1, d_2) =1$, for otherwise $\gcd(\hto(a), \hto(b)) \neq 1$, which contradicts the premise.  

Since $a$, $b$, $a \wedge b \in B$, Lemma \labelcref{lem: gcd} provides elements $e,f \in B$ with $\hto(e)=d_1$ and $\hto(f)=d_2$.  
Moreover, because $e,f \le a \wedge b$, it holds $e \not\parallel f$.  
Applying Lemma \labelcref{lem: gcd} once again, we obtain a term $t(x,y)$ such that $t(e,f)$ is an atom of $A$.

So, in any case, $\gcd(\hto_A(a), \hto_A(b)) = 1$ implies that an atom of $\al A$ is also in $\al B$. Let $e$ be an atom of $\al A$. Given $c \in B$ and $d \in A$ such that $d \leq c$, we have $d = c \ominus k e$ for some $k \in \mathbb{Z}^+$. Consequently, $\al B$ is a downset of $\al A$.
\end{proof}

\begin{theorem} \label{thm: sub of A}
Let $\al A$ be a finite subdirectly irreducible cBCK-algebra. Then, (up to isomorphism) the following holds 
$$\trop{S}(\al A) = \trop S_d(\al A) \cup \trop S_\delta(\trop S_d (\al A))\text{.}$$ 
Moreover, if $\trop S_\delta(\trop S_d(\al A))$ consists only of chains, then $\trop{IS}(\al A) = \trop{IS}_d(\al A)$.
\end{theorem}
\begin{proof}
The inclusion 
$$S(\al A) \supseteq \trop S_d(\al A) \cup \trop S_\delta(\trop S_d (\al A))$$ 
is evident. Let us prove the opposite inclusion. Let $\al B \in \trop S(\al A)$. We will show the implication
$$\al B \notin \trop S_\delta(\trop S_d (\al A)) \Rightarrow \al B \in \trop S_d(\al A).$$
This is sufficient, since it is then impossible for $\al B$ to lie outside both $\trop S_\delta(\trop S_d(\al A))$ and $\trop S_d(\al A)$.

Assume $\al B \notin \trop S_\delta(\trop S_d(\al A))$. 
If $\gcd(\hto_A(a), \hto_A(b)) = 1$ for some $a, b \in B$, then, by Proposition \labelcref{prop: gcdab}, $\al B \in \trop S_d(\al A)$. In this case, it also follows that $\gcd(\hto_A(b) \mid b \in B) = 1$. Otherwise, we must have $\gcd(\hto_A(b) \mid b \in B) = k > 1$, which implies $B = D_k$ for some $\al D \in \trop S_d(\al A)$. This is impossible since we assumed $\al B \notin \trop S_\delta(\trop S_d(\al A))$.

To prove the moreover part, let $\hto(\al A) = n$ and assume that $\trop S_\delta(\trop S_d(\al A))$ contains only chains. For any $\al C \in \trop S_\delta(\trop S_d(\al A))$, we have $\hto(\al C) = k \le n$. Since for every $i \in \{0, \dots, n\}$ there exists $a \in A$ with $\hto(a) = i$, there exists $a \in A$ such that $\hto(a) = k$. Then $\al{[0,a]_A} \in \trop S_d(\al A)$ is a finite chain of height $k$, as is $\al C$, so $\al{[0,a]_A} \cong \al C$.
\end{proof}

\begin{remark}
The opposite implication of the moreover part is not true, i.e. there are algebras $\al A$ such that $\trop{IS}(\al A) = \trop{IS}_d(\al A)$ and $\trop S_\delta(\trop S_d(\al A))$ does not consist of chains only. A simple counterexample is shown in Figure \labelcref{fig: example1}. 

\end{remark}

\begin{figure}[b]
    \centering
    \includegraphics[height= .18\textheight]{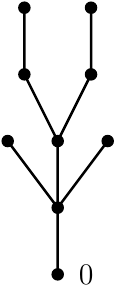}
    \caption{Example of cBCK-algebra satisfying $\trop{IS}(\al A) = \trop{IS}_d(\al A)$.}
    \label{fig: example1}
\end{figure}

Using $\trop S_\delta$ it is possible to provide a simple characterisation of the generators of $\al A$.
\begin{proposition}
Let $\al A$ be a finite subdirectly irreducible cBCK-algebra. Then, the smallest set generating $\al A$ is $ \meo(\al A)$ if and only if $\trop S_\delta(\al A) = \emptyset$.
\end{proposition}
\begin{proof}
($\Rightarrow$) Suppose that $\meo(\al A)$ generates $\al A$.  
Then there exists a BCK-term $t(x_1, \dots, x_n)$ such that $t(m_1, \dots, m_n) = a$, where $m_1, \dots, m_n \in \meo(\al A)$ and $a$ is an atom of $\al A$.  
Now, let $\al B \in S_\delta(\al A)$. By the definition of $\trop S_\delta(\al A)$, this means that $B = A_k \cup \meo(\al A)$ for some $k > 1$. Since $\meo(\al A) \subseteq B$, it follows that $a \in B$.  
However, $\al B$ cannot contain $a$ because $\hto_A(a) = 1$.

($\Leftarrow$) Let $S_\delta(\al A) = \emptyset$.  
Let $\al B$ be the subalgebra generated by the set $\meo(\al A)$. Then $\al B \in \trop S(\al A)$.  
By Theorem \labelcref{thm: sub of A}, we have $\al B \in \trop S_d(\al A)$ or $\al B \in \trop S_\delta(\trop S_d(\al A))$.  
Since $\meo(\al A) \subseteq \al B$, if $\al B \in \trop S_\delta(\trop S_d(\al A))$, it would follow that $\al B \in \trop S_\delta(\al A)$, contradicting the assumption that $S_\delta(\al A) = \emptyset$.  
Hence, $\al B \in \trop S_d(\al A)$.  
Clearly, $\al B$ is not trivial. Therefore, $\al B$ contains an atom of $\al A$. Since $\al B$ contains an atom of $\al A$ and $\meo(\al A)$, by Lemma \labelcref{lem: gen A}, we conclude that $\al B = \al A$.  

It is also evident that no element of $\meo(\al A)$ can be omitted, as $\ominus$ is decreasing. We conclude that $\meo(\al A)$ is the smallest set generating $\al A$.
\end{proof}

\section{Covers of finitely generated varieties}
Let us first clarify the following fact: for any subdirectly irreducible cBCK-algebra $\al A$ of infinite height, the variety $\vag{\al A}$ contains all finite chains $\al S_n$, $n \in \mathbb{N}$. Indeed, $\al A$ contains an infinite \L BCK-chain as a subalgebra, and any infinite \L BCK-chain generates the entire variety of \L BCK-algebras. Hence, all finite cBCK-chains are necessarily contained in $\vag{\al A}$.

\begin{proposition}\label{prop: cover is fin gen}
    Any cover of a finitely generated variety of cBCK-algebras is finitely generated.
\end{proposition}

\begin{proof}

Let $\var{V}$ be a finitely generated variety of cBCK-algebras, and let $\var{K}$ be a variety that is not finitely generated.  
We prove the contraposition, i.e., that $\var{K}$ cannot be a cover of $\var{V}$.  

Since $\var{V}$ is finitely generated, there exists some $n \in \mathbb{N}$ such that $\hto(\al A) \leq n$ for every $\al A \in \trop{Si}(\var{V})$.  
If $\var{K}$ contains a subdirectly irreducible algebra of infinite height, then it also contains $\al S_{n+1}$.  
It follows that $\var{V} \subsetneq \var{V} \vee \vag{\al S_{n+1}} \subseteq \var{K}$, which contradicts the assumption that $\var{K}$ is a cover of $\var{V}$.  
Therefore, $\var{K}$ cannot contain a subdirectly irreducible algebra of infinite height.

Clearly, if $\var{K}$ has infinitely many irredundant generators, then it cannot be a cover, since one can select a subset of these generators to form a variety that is strictly smaller than $\var{K}$ but strictly larger than $\vag{\al A}$.

It remains to investigate the case where $\var{K}$ contains an infinite subdirectly irreducible cBCK-algebra of finite height. 
Such an algebra must have infinitely many incomparable elements; that is, its width is infinite. 
Let us denote this algebra by $\al B$. 
Similarly to the previous paragraph, there exists $k \in \mathbb{N}$ such that $|m(\al A)| \leq k$ for every \hbox{$\al A \in \trop{Si}(\var{V})$.} 
Now, choose $k+1$ incomparable elements in $\al B$ to generate a subalgebra $\al C$. 
The algebra $\al C$ is finite, yet $\al C \notin \var{V}$. 
Hence, $\var{V} \subsetneq \var{V} \vee \vag{\al C} \subsetneq \var{K}$.
\end{proof}

The previous result can be understood through the structural properties of subdirectly irreducible cBCK-algebras.  
Alternatively, it can also be seen from the existence of certain identities that bound the height and width of the subdirectly irreducible algebras.  
We present these identities here for the reader's insight, even though they are not strictly necessary for the argument.

The following is a restated version of \cite[Theorem 2.1]{Romanowska} and \cite[Proposition 3]{PalRom}.

\begin{samepage}
\begin{theorem}
Let $\al A$ be a subdirectly irreducible cBCK-algebra. Then, $\al A$ satisfies
$$x \ominus (n+1) y = x \ominus ny\text{,}$$
iff $\hto(\al A) \leq n$. Further, $\al A$ satisfies
$$ \bigwedge_{1 \leq i \neq j \leq n} (x_i \ominus x_j) = 0\text{,}$$
iff width of $\al A$ is $\leq n$. If $\al A$ has a finite height, then it is equivalent to $\vert \meo (\al A)\vert \leq n$.
\end{theorem}
\end{samepage}

We have seen that the cover cannot contain a subdirectly irreducible algebra of infinite height or infinite width. 
Of course, these restrictions can be made tighter.

\begin{proposition}
    Let $\al A$ be a finite subdirectly irreducible cBCK-algebra, let $\var{K}$ be a cover of $\vag{\al A}$ and let $\al B \in \trop{Si}(\var{K})$. Then
    \begin{enumerate}[label=(\roman*)]
        \item $\vert \hto(\al B)\vert \leq \hto(\al A) +1$,
        \item $\vert \meo(\al B)\vert \leq \vert\meo(\al A)\vert + 1$.
    \end{enumerate}
\end{proposition}
\begin{proof}
The arguments proceed similarly to the proof of Proposition \labelcref{prop: cover is fin gen}.  
Let $\al A$ be a non-trivial finite subdirectly irreducible cBCK-algebra.  
Furthermore, let $\var{K}$ be a variety, and let $\al B \in \trop{Si}(\var{K})$.
    
(i) Suppose that $\vert \hto(\al B)\vert > \hto(\al A) + 1$.  
Then $\al B$ contains a chain $\al S_{\hto(\al A)+2}$ of height $\hto(\al A)+2$.  
Since $\al S_{\hto(\al A)+1} \notin \vag{\al A}$ and $\vag{\al S_{\hto(\al A)+1}} \subsetneq \vag{\al S_{\hto(\al A)+2}}$, it follows that $\vag{\al A} \subsetneq \vag{\al A} \vee \vag{\al S_{\hto(\al A)+1}} \subsetneq \var{K}$. Hence, $\var{K}$ is not a cover.

(ii) If $\al B$ does not have a finite height, then by (i) we know that $\var{K}$ cannot be a cover, and the argument is complete.  
Thus, assume that $\al B$ has a finite height.  
Then its width is exactly the number of elements in $\meo(\al B)$. 
Suppose that $\vert \meo(\al B) \vert > \vert \meo(\al A) \vert + 1$.  
It follows that $\al B$ contains a subalgebra, say $\al C$, such that $\vert \meo(\al C) \vert = \vert \meo(\al A) \vert + 1$; for instance, we may choose a suitable subset of $\meo(\al B)$ to generate $\al C$.  
Once again, $\vag{\al A}$ cannot contain $\al C$, and moreover $\vag{\al C} \subsetneq \var{K}$ (since $\al C$ has a strictly greater width).  
Therefore, $\vag{\al A} \subsetneq \vag{\al A} \vee \vag{\al C} \subsetneq \var{K}$.
\end{proof}

\begin{definition}
    Let $\var{V}$ be a variety of cBCK-algebras. We say that $\var{V}$ is $n$-generated if and only if there is $n$-element set of mutually non-isomorphic subdirectly irreducible algebras of $\var{V}$ that generate $\var{V}$, and if there is no other set of subdirectly irreducible algebras with the same property and smaller cardinality.
\end{definition}
\begin{remark}
If $\var{V}$ is $n$-generated, there can be more than one set of subdirectly irreducible algebras witnessing this fact.  
For example, for the variety of all \L BCK-algebras $\var{L}$ we have
$\var{L} = \vag{\mathbb{Z}^+} = \vag{\al{[0,1]}_\mathbb{R}}$.
Thus, $\var{L}$ is $1$-generated, yet there exist at least two different generating subdirectly irreducible algebras.

However, for finitely generated varieties of cBCK-algebras, this cannot happen. If $\var{V}$ is a finitely $n$-generated variety, then there exists a unique (up to isomorphism) $n$-element set of mutually non-isomorphic subdirectly irreducible algebras generating $\var{V}$.  
Indeed,
$$
\trop{Si}(\var{V})
   = \trop{Si}\bigl(\vag{\al A_1,\dots,\al A_n}\bigr)
   = \trop{Si}(\vag{\al A_1}) \cup \dots \cup \trop{Si}(\vag{\al A_n}),
$$
and this set consists only of (isomorphic copies of) subalgebras of
$\al A_1,\al A_2,\dots,\al A_n$. 
\end{remark}

\begin{lemma}
Let $\var{V}$ be a finitely $n$-generated variety of cBCK-algebras and let $\var{W}$ be its covering. Then $\var{W}$ is at most $(n+1)$-generated.
\end{lemma}
\begin{proof}
Let $\var{V}$ be $n$-generated, i.e., $\var{V} = \vag{\al A_1, \dots, \al A_n}$, where $\trop{IS}(\al A_i) \not\subseteq \trop{IS}(\al A_j)$ for any two distinct $i,j \in \{1, \dots, n\}$. Similarly, let $\var{W}$ be at least $m$-generated, with $m \geq n+2$. Thus, $\var{W} = \vag{\al B_1, \dots, \al B_m}$, where $\trop{IS}(\al B_i) \not\subseteq \trop{IS}(\al B_j)$ for any two distinct $i,j \in \{1, \dots, m\}$. Furthermore, assume that $\var{V} \subseteq \var{W}$. We will show that $\var{W}$ cannot be a cover.  

It follows that
$$
\bigcup_{i=1}^n \trop{IS}(\al A_i) \subseteq \bigcup_{j = 1}^{m} \trop{IS}(\al B_j)\text{.}
$$
For each $\al A_i$, there exists at least one $\al B_j$ such that $\al A_i \in \trop{IS}(\al B_j)$. From the set of $\al B_j$, $j\in \{1, \dots, m\}$, we can choose a subset of elements $\al B_k$, $k \in K$, with $|K| = n$, such that for every $\al A_i$ there is at least one $\al B_k$ for which $\al A_i \in \trop{IS}(\al B_k)$.
Indeed, for each $\al A_i$ we pick just one $\al B_j$ from the set of all possible $\al B_j$ satisfying this condition. This gives a mapping from an $n$-element set, so the image contains at most $n$ elements.  

Furthermore, since $m \geq n+2$, there exist $\al B_{j_1}$ and $\al B_{j_2}$ such that $j_1 \neq j_2$ and $j_1 \notin K$, $j_2 \notin K$. Then
$$
\var{V} \subseteq \vag{\al B_k \mid k \in K} 
\subsetneq \vag{\{\al B_k \mid k \in K\} \cup \{\al B_{j_1}\}} 
\subsetneq \vag{\{\al B_k \mid k \in K\} \cup \{\al B_{j_1}, \al B_{j_2}\}} 
\subseteq \var{W}\text{.}
$$
Hence, $\var{W}$ is not a cover.
\end{proof}

We have bounds on the number of irredundant generators for covers. Next, we describe more precisely the elements that such a set may contain.

\begin{lemma}\label{lem: gen of cover}
    Let $\var{V}= \vag{\al A_1, \dots, \al A_n}$ be a finitely $n$-generated variety and let $\var{K} = \vag{\al B_1, \dots , \al B_m}$ be a finitely $m$-generated variety such that $\var{V} \subsetneq \var{K}$. Let $C$ be a subset of the set $\{\al B_1, \dots , \al B_m\}$ such that $\al B_i \in C$ iff for all $j \in \{1, \dots,n \}$ it holds $\al B_i \not\cong \al A_j$.
    If $\vert C \vert >1$, then $\var{K}$ is not a cover.
\end{lemma}
\begin{proof}
Let $\al B_{i_1}$, $\al B_{i_2} \in C$ be distinct. Let $\trop{Prop}(\al B_i)$ denote a set containing all $\al A_j$, $j \in \{1, \dots, n\}$, for which $\al A_j$ is embeddable to $\al B_i$ only as a proper subalgebra, i.e.,
$$\trop{Prop}(\al B_i)= \{\al A \in \{\al A_1, \dots, \al A_n\} \mid  \al A \in \trop{IS}(\al B_i), \, \al A \not\cong \al B_i\}.$$
If $\trop{Prop}(\al B_{i_1}) = \trop{Prop}(\al B_{i_2})= \emptyset$. Then, since $\al B_{i_1}$, $\al B_{i_2} \in C$, we have
$$\var{V} \subsetneq \vag{\al B_i \mid i \in \{1, \dots, n\}, \, i \neq i_1} \subsetneq \var{K}.$$
Thus $\var{K}$ is not a cover.

Without loss of generality, assume now that $\trop{Prop}(\al B_{i_1}) \neq \emptyset$. Then
$$\var{V} \subsetneq \vag{\{\al B_i \mid i \in \{1, \dots, n\}, \, i \neq i_1\} \cup \trop{Prop}(\al B_1)} \subsetneq \var{K}.$$
That yields once again that $\var{K}$ is not a cover.
\end{proof}

In the following statement, the set $\trop{Prop}(\al B)$ has the same meaning as in the proof above.

\begin{proposition}\label{prop: cover has a form}
Let $\var{V} = \vag{\al A_1, \dots, \al A_n}$ be a finitely $n$-generated variety, and let $\var{K}$ be a cover of $\var{V}$.  
Then there exists $\al B \in \trop{Si}(\var{K})$ such that  
$\bigl(\{\al A_1, \dots, \al A_n\} \setminus \trop{Prop}(\al B)\bigr) \cup \{\al B\}$  
is an irredundant set of generators for $\var{K}$.  
Thus
$$
   \var{K}
   = \vag{\bigl(\{\al A_1, \dots, \al A_n\}\setminus \trop{Prop}(\al B)\bigr) \cup \{\al B\}}
$$
is $m$-generated, where $m = n + 1 - \lvert \trop{Prop}(\al B) \rvert$.
\end{proposition}

\begin{proof}
Let $\var{K}$ be a cover of $\var{V}$.  
By Lemma~\labelcref{lem: gen of cover}, there is exactly one algebra $\al B$ in an irredundant generating set of $\var{K}$ that is not isomorphic to any of $\al A_1, \dots, \al A_n$.  If $\trop{Prop}(\al B) = \emptyset$, then the set $\{\al A_1, \dots, \al A_n, \al B\}$ is an irredundant set of generators for $\var{K}$, because no $\al A_i$ lies in $\trop{IS}(\al B)$. If $\trop{Prop}(\al B) \neq \emptyset$, then the generating set for $\var{K}$ is  
$\bigl(\{\al A_1, \dots, \al A_n\}\setminus \trop{Prop}(\al B)\bigr) \cup \{\al B\}$,  
which is clearly irredundant.
\end{proof}

We now move on to the first step in the construction of the covers. Let $\al A$ be a finite subdirectly irreducible cBCK-algebra.
For any $\al{B} \in \trop{S}(\al{A})$ and any $a \in \al{B}$ with $a \neq 0$, we define a tree $\al{B}_a$ as follows:

\begin{enumerate}
    \item Let $c \notin B$ be a new element and $B_a := B \cup \{c\}$.
    \item Put $a \prec c$.
\end{enumerate}

As stated in Corollary~\labelcref{cor: trees are cbck}, there is (up to isomorphism) a unique cBCK-algebra\footnote{We use the same symbol for the algebra because, within the isomorphism class, we can choose a representative with the identical underlying set and tree structure.} $\al B_a$ whose underlying semilattice is this tree.  
What we have just constructed is a new subdirectly irreducible algebra that depends on the choice of $a \in A$ and differs from the original by the single additional element $c$ attached directly to $a$.

Although it is not required for the subsequent arguments, we can describe explicitly how the operations of the new algebra $\al B_a$ extend those of the original algebra.  
For completeness, we provide this description in the next paragraph.

Denote $\al B = (B, \ominus, 0)$ and $\al B_a = (B_a, \ominus_a, 0)$.  The operation $\ominus_a$ coincides with $\ominus$ in $B \times B$.  To extend it to $B_a$, we first put $c \ominus_a c := 0$. Next, let $d$ be an atom of $\al B$.  
For any $b \in B$, define
\begin{align*}
b \ominus_a c &:= b \ominus a,\\
c \ominus_a b &:= a \ominus \bigl(\hto_B(a \wedge b) - 1\bigr)d.
\end{align*}

To see that the above is indeed a well-defined cBCK-algebra, let us  state the following lemma.
\begin{lemma}
Let $\al A$ be a finite subdirectly irreducible cBCK-algebra with atom $a$.
For any $b$, $c \in A$,
$$
b \ominus c = b \ominus \hto(b \wedge c)\, a.
$$
\end{lemma}

\begin{proof}
Both $b \ominus c$ and $b \ominus \hto(b \wedge c) a$ lie in the finite chain $[0,b]$, where each element is uniquely determined by its height.  
It therefore suffices to show that these two elements have the same height.
By repeated use of Lemma~\labelcref{lem: height and minus},
\begin{align*}
\hto\bigl(b \ominus \hto(b \wedge c)\, a\bigr)
  &= \hto(b) - \hto(b \wedge c)  \\
  &= \hto\bigl(b \ominus (b \wedge c)\bigr) \\
  &= \hto(b \ominus c).
\end{align*}
Since they have equal height, the two elements are equal.
\end{proof}

With the aid of the above lemma, verifying that $\al B_a$ satisfies the identities \labelcref{id: BCK2}, \labelcref{id: BCK3}, \labelcref{id: commutativity}, and \labelcref{id: BCKex} is straightforward.  
Consequently, $\al B_a$ is a well-defined cBCK-algebra, and we omit the details of this routine computation.

\vspace{\baselineskip}
We now proceed to the second step in the construction of the covers.
Let $\var{V} = \vag{\al A_1, \dots, \al A_n}$ be a finitely $n$-generated variety of cBCK-algebras. For every $\al A_i$, $i \in \{1, \dots, n\}$, consider $\al B \in \trop S(\al A_i)$ and constructed $\al B_a$, where $a \in A_i$, $a \neq 0$. Furthermore, we denote
$$
\trop S(\al B_a)(c) = \{\al D \in \trop S(\al B_a) \mid c \in D\}.
$$

We are interested in the following set:
$$
\{\al C \in \trop S(\al B_a)(c) \mid \al C \notin \trop{IS}(\al A_i)\text{, for any } i \in \{1,\dots, n\}\}.
$$
This set may be empty for some choice of $\al B$ and $a \in A_i$. However, if it is not empty, it can be regarded as a subposet of the lattice of subalgebras $\trop S(\al B_a)$. This subposet does not necessarily have the smallest element, but it does have some minimal elements. Let us denote them as
$$
C^i_{\al B_a} = \min\{\al C \in \trop S(\al B_a)(c) \mid \al C \notin \trop{IS}(\al A_i)\text{, for any } i \in \{1,\dots, n\}\}.
$$

In the following, we will prove that these minimal elements are in the generating sets of covers. To ensure that we account for all possibilities and thus obtain all covers, we consider the set
$$
\trop{Cov}(\var{V}) = \bigcup \{ C^i_{\al B_a} \mid i \in \{1, \dots, n\}, \, \al B \in \trop S(\al A_i),\, a \in A_i \}.
$$

\begin{proposition}\label{prop: cons gen cov}
Let $\var{V} = \vag{\al A_1, \dots, \al A_n}$ be a finitely $n$-generated variety of cBCK-algebras.  
For any $\al C \in \trop{Cov}(\var{V})$, the variety $\var{V} \vee \vag{\al C}$ is a cover of $\var{V}$.
\end{proposition}

\begin{proof}
Let $\var{V} = \vag{\al A_1, \dots, \al A_n}$ be a finitely $n$-generated variety of cBCK-algebras.  
Let $\var{K}$ be a cover of $\var{V}$ such that $\var{V} \subsetneq \var{K} \subseteq \var{V} \vee \vag{\al C}$.  
It follows that
\begin{align*}
    \bigcup_{i=1}^n \trop{IS}(\al A_i) \subsetneq \trop{Si}(\var{K}) \subseteq \bigcup_{i=1}^n \trop{IS}(\al A_i) \cup \trop{IS}(\al C).
\end{align*}
From this we conclude that for any $\al K \in \trop{Si}(\var{K})$, either $\al K \in \trop{IS}(\al A_i)$ for some $i$, or $\al K \notin \trop{IS}(\al A_i)$ for all $i \in \{1, \dots, n\}$.

Suppose that $\al K \notin \trop{IS}(\al A_i)$ for any $i$.  
Then $\al K \in \trop{IS}(\al C)$. Without loss of generality, assume that $\al K \in \trop{S}(\al C)$.
Since $\al C \in \trop{Cov}(\var{V})$, we have $\al C \in \trop{S}(\al B_a)(c)$ and $\al C \notin \trop{IS}(\al A_i)$ for any $i$.  
There are two possibilities: Either $c \in \al K$, or $c \notin \al K$.

If $c \in \al K$, then necessarily $\al K = \al C$, because both $\al K$ and $\al C$ satisfy the condition of not being in any $\trop{IS}(\al A_i)$, and $\al C$ is defined as minimal with this property.

If $c \notin \al K$, then from the construction of $\al C$ we have $\al C = C' \cup \{c\}$.  
Moreover, $C'$ is the universe of a subalgebra of $\al C$, and $\al K \subseteq \al C'$.  
But $\al C' \in \trop{IS}(\al A_i)$ for some $i$, so $\al K \in \trop{IS}(\al A_i)$ as well, which contradicts the assumption that $\al K \notin \trop{IS}(\al A_i)$ for any $i$.

Consequently, the only possibility is that $c \in \al K$, in which case $\al K = \al C$.  
Hence, every $\al K \in \trop{Si}(\var{K})$ satisfies either $\al K \in \trop{IS}(\al A_i)$ for some $i$ or $\al K \cong \al C$.  
Thus, $\var{K} = \vag{\al A_1, \dots, \al A_n, \al C}$.
\end{proof}

We have shown that the given construction indeed produces a cover.  
However, the crucial question is the converse. Does every cover arise from this construction?  
The answer is affirmative, and this forms the main theorem of the paper.

\begin{theorem}[Main Theorem]
Let $\var{V} = \vag{\al A_1, \dots, \al A_n}$ be a finitely $n$-generated variety of cBCK-algebras.  
Let $\var{K}$ be a cover of $\var{V}$.  
Then there exists $\al C \in \trop{Cov}(\var{V})$ such that $\var{K} = \var{V} \vee \vag{\al C}$.
\end{theorem}

\begin{proof}
By Proposition \labelcref{prop: cover has a form}, 
$\var{K} = \vag{\big(\{\al A_1, \dots, \al A_n\}\setminus \trop{Prop}(\al B)}\big) \cup \{\al B\}$ 
for some $\al B \in \trop{Si}(\var{K})$, and this set of generators is irredundant.

Let $c \in \meo(\al B)$. Then $B \setminus \{c\}$ is the universe of a subalgebra of $\al B$, since it forms a downset of $\al B$. 
Assume that $\al{B \setminus \{c\}} \notin \trop{IS}(\al A_i)$ for any $i \in \{1, \dots, n\}$. 
Then $\al{B \setminus \{c\}} \subsetneq \al B$ and
$$
\var{V} \subsetneq \vag{\al A_1, \dots, \al A_n, \al{B \setminus \{c\}}} \subsetneq \var{K},
$$
because $\al B$ is not contained in $\vag{\al A_1, \dots, \al A_n, \al{B \setminus \{c\}}}$. 
This contradicts the assumption that $\var{K}$ is a cover. 
Hence, the opposite must hold.

Let $B \setminus\{c\} \in \trop{IS}(\al A_i)$ for some $i$. 
Then $\al B$ can be viewed as $\al C_a$ constructed from $\al C \cong \al{B \setminus\{c\}}$, 
where $\al C \in \trop{S}(\al A_i)$. 
Thus, we have $\al B \cong \al C_a$ with $C_a = C \cup \{d\}$, 
where $d$ is a new element attached to $a$. 
Since $\al B \notin \trop{IS}(\al A_i)$ for any $i$, the same applies to $\al C_a$. We claim that $\al C_a$ is minimal with this property and containing $d$. 
Indeed, if $\al C_a$ was not minimal, there would exist $\al D \subsetneq \al C_a$ such that
$$
\var{V} \subsetneq \vag{\al A_1, \dots, \al A_n, \al D} \subsetneq \var{K},
$$
which contradicts the fact that $\var{K}$ is a cover.

Since $\al C_a$ is minimal, we have $\al C_a \in \trop{Cov}(\var{V})$, 
which, together with $\al B \cong \al C_a$, gives
{\setlength{\belowdisplayskip}{0pt}
$$
\var{K} = \vag{\al A_1, \dots, \al A_n, \al C_a} = \var{V} \vee \vag{\al C_a}.
$$}
\end{proof}

Let us summarise the results of Proposition \labelcref{prop: cons gen cov} and the main theorem in the following corollary.

\begin{corollary}
Let $\var{V} = \vag{\al A_1, \dots, \al A_n}$ be a finitely $n$-generated variety of cBCK-algebras, 
and let $\var{K}$ be a variety of cBCK-algebras. 
Then $\var{K}$ is a cover of $\var{V}$ if and only if 
$\var{K} = \var{V} \vee \vag{\al C}$ for some $\al C \in \trop{Cov}(\var{V})$.
\end{corollary}

From a practical perspective, we should note that if we have, say, two $1$-generated varieties $\vag{\al A_1}$ and $\vag{\al A_2}$, 
then, if we find a cover $\var{K}$ of $\vag{\al A_1}$ and this cover does not coincide with $\vag{\al A_2}$, 
it follows from the congruence distributivity that $\vag{\al A_2} \vee \var{K}$ is a cover of $\vag{\al A_1, \al A_2}$.

In other words, many covers of a finitely $n$-generated variety $\vag{\al A_1, \dots, \al A_n}$ can be obtained by first considering the covers of $\vag{\al A_i}$ for $i \in \{1, \dots, n\}$. 
However, it should be noted that this does not necessarily produce all the covers of $\vag{\al A_1, \dots, \al A_n}$.

\begin{remark}\label{rem: last}
In the first section, we described subalgebras of finite subdirectly irreducibles. In some cases, we may employ the findings to make the construction of cover more effective. Let us consider $1$-generated variety $\vag{\al A}$. In the first step of the construction we choose arbitrary $\al B \in \trop{S}(\al A)$. However, we claim that it is enough to consider $\al B \in \trop S_\delta(\trop S_d(\al A)) \cup \{\al A\}$ to obtain all the covers. The reason is that if $\al B \in \trop{S}_d(\al A)$. Then for any $a \in A$, $a \neq 0$, the following holds
\begin{align*}
    C_{\al B_a} = \min\{\al C \in \trop S(\al B_a)(c) \mid \al C \notin \trop{IS}(\al A)\} \subseteq C_{\al A_a} = \min\{\al C \in \trop S(\al A_a)(c) \mid \al C \notin \trop{IS}(\al A)\}.
\end{align*}
Indeed, since $\trop{S}(\al B_a)(c) \subseteq \trop{S}(\al A_a)(c)$, if $\al C \in C_{\al B_a}$ then $\al C \in \trop{S}(\al A_a)$.

There may be special cases for which 
$\trop S_{\delta}(\trop S_d (\al A)) \subseteq \trop{IS}_d(\al A)$; see Theorem \labelcref{thm: sub of A}. 
In such cases, the first step of the construction may proceed simply by taking $\al B = \al A$.
\end{remark}

Let us finish the article by applying the general theorem to specific examples.
Before presenting these examples, we first introduce a notion useful for describing more complex cBCK-algebras.

Let $P = p_1, \dots, p_k$ be a list of natural numbers with $k \geq 2$. 
We denote by $\trop M_P(\al S_n)$ the cBCK-algebra obtained by gluing $S_{p_1}, \dots, S_{p_k}$ on top of $\al S_n$ (see Figure \labelcref{fig: MP}). 
The operations are defined so that for any $i \in \{1, \dots, k\}$, the maximal chain $\al{[0, m_i]}$ is isomorphic to $\al S_{n+p_i}$. 
In the special case where $p_1 = p_2 = \dots = p_k = 1$, we abbreviate $\trop M_P(\al S_n)$ as $\trop M_k(\al S_n)$. 
For example, $\trop M_2(\al S_n)$ is a cBCK-algebra with two maximal chains, each isomorphic to $\al S_{n+1}$.

\begin{figure}[t]
    \centering
    \includegraphics[height= .2\textheight]{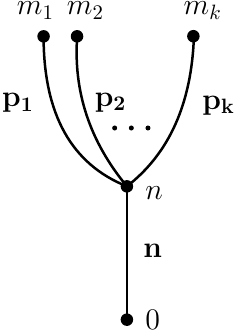}
    \caption{The visualization of $\trop M_P(\al S_n)$. Bold letters refer to lengths of intervals.}
    \label{fig: MP}
\end{figure}

\begin{example}\label{exmp: Sn}
Let $\al A = \al S_n$. Then $\trop S_\delta(\trop S_d(\al A)) \subseteq \trop{IS}_d(\al A)$. 
Therefore, based on the arguments in the previous remark, it is sufficient to consider $\al B = \al A$ in the construction of the cover. 
Hence, all the covers of $\vag{\al S_n}$ are:
\begin{itemize}
    \item $\vag{\al{S}_{n+1}}$,
    \item $\vag{\al S_n} \vee \vag{\trop M_2(\al S_m)}$, $1 \leq m \leq n-1$.
\end{itemize}
\end{example}

\begin{example}
Let $\al A = \trop M_P(\al S_q)$, where $P = k \times p$ (i.e., $P$ is the list of length $k$ consisting of $p$ repeated $k$ times). 
Further, let $p$ and $q$ be such that for any $p' \in \{1, \dots, p\}$, $p'$ and $q$ are coprime.\footnote{This condition forces $p < q$.} 
Then, again, $\trop S_\delta(\trop S_d(\al A)) \subseteq \trop{IS}_d(\al A)$, 
and one may easily check that all the covers of $\vag{\trop M_P(\al S_q)}$ are:
\begin{itemize}
    \item $\vag{\trop M_P(\al S_q)} \vee \vag{\trop M_{k+1}(\al S_q)}$,
    \item $\vag{\trop M_P(\al S_q)} \vee \vag{\al S_{q+p+1}}$,
    \item $\vag{\trop M_P(\al S_q)} \vee \vag{\trop M_2(\al S_m)}$, $1 \leq m \leq q-1$ or $q+1 \leq m \leq q+p-1$.
\end{itemize}
This includes, in particular, the case $\trop{M_k}(\al S_n)$ for $n, k \in \mathbb{N}$, $k \geq 2$. 
Example \labelcref{exmp: Sn} and the case $\al A = \trop{M_k}(\al S_n)$ were previously illustrated in unpublished notes by Jan Kühr and Petr Ševčík. 
However, they only demonstrate what the covers are without presenting any general method to find them.

\end{example}

\section{Conclusion}

In this article, we have described a method for constructing any cover of a finitely $n$-generated variety of cBCK-algebras. Consequently, this provides a way to understand the lower part of the lattice of subvarieties of cBCK-algebras, which may be relevant for future research. Hereditary simplicity is the key to making this construction work. Although this property is quite special, it is natural to ask whether a similar construction could be applied in settings beyond cBCK-algebras.

For instance, it is known that some subvarieties of effect algebras also possess this property; see \cite{JendaEf} for more details. In this context, simple algebras are horizontal sums of MV-algebras. Although horizontal sums of MV-chains are not trees, it may still be possible to consider constructions that extend these algebras in an analogous manner.

Several recent works on BCK-algebras address the structure of subvariety lattices; 
for example, \cite{Agliano} investigates splittings in a wide range of algebraic structures related to BCK-algebras or derived from them, 
though not specifically in commutative BCK-algebras. 
To our knowledge, the only explicit result in this setting is due to Romanowska \cite{Romanowska}, 
who proved that any variety generated by a finite chain is splitting in the subvariety lattice of all cBCK-algebras. 
Beyond this, the structure of splitting varieties within commutative BCK-algebras remains incomplete. 

The aim of this article has been to describe the covers of finitely generated varieties of cBCK-algebras. 
These results may provide useful insights for identifying join-irreducible elements near the bottom of the lattice, 
which in turn could serve as candidates for splitting algebras. 
Further examination of this approach could be considered.

In addition to theoretical considerations, the construction of covers can also be approached from an algorithmic perspective.
In particular, one could envision a program that computes all the covers (i.e., their irredundant generators) for a given finitely $n$-generated variety. 
Such a tool would be useful for studying the lattice of subvarieties, and Remark \labelcref{rem: last} together with Theorem \labelcref{thm: sub of A} provides natural strategies for optimising such an implementation.

\nocite{CornishBCK}

\end{document}